\documentclass[11pt]{article}
\usepackage{amsmath,amsfonts,amsthm,amssymb,mathtools}
\usepackage{mathrsfs,graphicx,color,latexsym,tikz,calc}
\usepackage[colorlinks,bookmarksopen,bookmarksnumbered,citecolor=blue,linkcolor=red,urlcolor=blue]{hyperref}
\usepackage{enumitem}
\usepackage{authblk}
\usetikzlibrary{shadows}
\usetikzlibrary{patterns,arrows,decorations.pathreplacing}

\newtheorem{definition}{Definition}
\newtheorem{theorem}{Theorem}
\newtheorem{lemma}{Lemma}

\newtheorem{conjecture}{Conjecture}

\newtheorem{prop}{Proposition}

\newcommand{\spl}{\operatorname{spl}}
\newcommand{\Nmono}{N_{\mathrm{mono}}}

\begin{document}

\title{Ramsey multiplicity and extremal colorings for odd cycles}
\author{Ting HUANG$^1$, Junying LU$^{2,}$\footnote{Corresponding author. Email: junyinglu@njust.edu.cn}~, Jiabao YANG$^1$, Yaojun CHEN$^1$\\
 {\small{$^1$School of Mathematics, Nanjing University, Nanjing 210093, P.R. CHINA}}\\
 {$^2$\small{School of Mathematics and Statistics, Nanjing University of Science and Technology,\\
Nanjing 210094, P.R. CHINA}}\\
 }
\date{ }

\maketitle

\begin{abstract}
The Ramsey number $r(H)$ of a graph $H$ is the minimum positive integer $N$ such that every red/blue edge-coloring of the complete graph $K_N$ on $N$ vertices contains a monochromatic copy of $H$.  
The Ramsey multiplicity $M(H,n)$ is the minimum number of monochromatic copies of $H$ over all red/blue edge-colorings of $K_n$.
It is called threshold Ramsey multiplicity if $n=r(H)$, and denoted by $m(H)$.
The only previously known general infinite family for which $m(H)$ has been determined is stars, due to Harary and Prins (1974). Let $C_k$ denote a cycle on $k$ vertices. Conlon, Fox, Sudakov, and Wei (2022) conjectured that $m(C_k)=(k-1)!/2$ for every sufficiently large odd integer $k$. 
In this paper,  we determine $M(C_k,r(C_k)+\ell)$ for every fixed nonnegative integer $\ell$ and all sufficiently large odd $k$, and characterize all extremal colorings,  thereby confirming the conjecture.
This is also a second general infinite family for which $m(H)$ has been determined.

\end{abstract}

\baselineskip=0.202in

\section{Introduction}\label{sec-intr}

The Ramsey number $r(H)$ of a graph $H$ is the least positive integer $n$
such that every red/blue coloring of the edges of $K_n$ contains a
monochromatic copy of $H$. Although Ramsey's theorem guarantees that
$r(H)$ is finite, exact values are known for only a few general families.
Two classic examples are paths and cycles. Gerencs\'er and
Gy\'arf\'as~\cite{Gerencser1967} determined the Ramsey numbers of paths,
while Rosta~\cite{Rosta1973} and, independently, Faudree and
Schelp~\cite{Faudree1974} determined the Ramsey numbers of cycles. In
particular, let $C_k$ denote the cycle on $k$ vertices.

\begin{theorem}[Faudree and Schelp \cite{Faudree1974}, Rosta \cite{Rosta1973}]\label{thm:Ramsey-cycles}
For every odd integer $k\geq 5$,
$r(C_k)=2k-1$. 
\end{theorem}

The Ramsey number $r(H)$ records only the smallest order at which a
monochromatic copy of $H$ becomes unavoidable. At this threshold,
however, a coloring may be forced to contain many monochromatic copies of $H$. To measure this phenomenon, for a red/blue edge-coloring $\chi: E(K_n)\to\{\mathrm{red},\mathrm{blue}\}$, let $\Nmono(H,\chi)$ denote the total number of monochromatic copies of $H$ in $\chi$.

For a graph $H$ and a positive integer $n$, define
\[
M(H,n)=
\min_{\chi:E(K_n)\to\{\mathrm{red},\mathrm{blue}\}}
\Nmono(H,\chi).
\]
Thus $M(H,n)$ is the minimum number of monochromatic copies of $H$
over all red/blue edge-colorings of $K_n$. By the definition of the
Ramsey number, $M(H,n)=0$ if and only if $n<r(H)$.

\begin{definition}[Threshold Ramsey multiplicity]
The \emph{threshold Ramsey multiplicity} of a graph $H$ is the first
positive value of $M(H,n)$, that is,
\[
m(H):=M(H,r(H)).
\]
\end{definition}

Thus $r(H)$ asks when the first monochromatic copy must occur, whereas
$m(H)$ asks how many copies must occur at precisely that threshold.
Let $P_k$ denote the path on $k$ vertices and $K_{1,k}$ denote the star with one central vertex and $k$ leaves.
Harary and Prins~\cite{HararyPrins} initiated the systematic study of
this parameter and asked, in particular, for the determination of
$m(P_k)$ and $m(C_k)$. And they determined the exact value of $m(K_{1,k})$. This is the only previously known general infinite family for which $m(H)$ has been determined.
\begin{theorem}[Harary and Prins \cite{HararyPrins}]
For every integer $k\geq1$,
\[
m(K_{1,k})=
\begin{cases}
1,  & \text{if $k=1$ or $k$ is even},\\[1mm]
2k, & \text{if $k\geq3$ is odd}.
\end{cases}
\]
\end{theorem}
This shows that threshold Ramsey multiplicity is sensitive even for
simple graph families and is considerably subtler than the corresponding
Ramsey number.


The threshold Ramsey multiplicity of odd cycles has been extensively studied. Harary and Prins \cite{HararyPrins} determined $m(C_3)=2$. This isolated case,
however, does not indicate the asymptotic behavior for long odd cycles.
For odd cycles of arbitrary length, the first general progress was made
by Rosta and Sur\'anyi~\cite{RostaSuranyi}. They proved that $m(C_k)$ grows
at least exponentially with $k$.

\begin{theorem}[Rosta and Sur\'anyi \cite{RostaSuranyi}]
There is an absolute constant $c>0$ such that, for every sufficiently
large odd integer $k$,
\[
m(C_k)\geq 2^{ck}.
\]
\end{theorem}

The first improvement beyond the exponential scale was obtained by
Rosta in an unpublished work. 
It was explicitly recorded by K\'arolyi and
Rosta~\cite{KarolyiRosta} and later by Conlon, Fox, Sudakov and
Wei~\cite{ConlonFoxSudakovWei}.

\begin{theorem}[Rosta]
As $k$ tends to infinity through the odd integers,
\[
m(C_k)^{1/k}\longrightarrow\infty.
\]
Equivalently, $m(C_k)$ grows
superexponentially in $k$.
\end{theorem}

K\'arolyi and Rosta~\cite{KarolyiRosta} subsequently obtained the first quantitative
superexponential lower bound. Their result gives an explicit positive
coefficient in front of the leading term $k\log k$ in the logarithm of
the multiplicity.

\begin{theorem}[K\'arolyi and Rosta \cite{KarolyiRosta}]
As $k$ tends to infinity through the odd integers,
\[
m(C_k)\geq
k^{(1/20-o(1))k}.
\]
\end{theorem}

In particular, the preceding theorem implies a lower bound of the form
$m(C_k)\geq k^{ck}$ for some absolute constant $c>0$. Conlon, Fox,
Sudakov and Wei~\cite{ConlonFoxSudakovWei} considerably strengthened this result. Their theorem
raises the coefficient of $k\log k$ in the exponent from $1/20$ to the
optimal value $1$.

\begin{theorem}[Conlon, Fox, Sudakov, and Wei
\cite{ConlonFoxSudakovWei}]\label{thm:CFSW-lower}
There is an absolute constant $c>0$ such that, for every odd integer
$k\geq 3$,
\[
m(C_k)\geq (ck)^k.
\]
\end{theorem}

The following natural construction establishes an upper bound of $m(C_k)$.
Two colorings are called \emph{equivalent} if one can be obtained from the other by relabeling the vertices and, possibly, exchanging the two colors. Let $\chi_0(a,b)$ denote the red/blue
edge-coloring of $K_{a+b}$ obtained from a vertex partition $A\cup B$,
where $|A|=a$ and $|B|=b$, by coloring all edges within $A$ and within
$B$ in one color and all edges between $A$ and $B$ in the other color.
Let $\chi_1(a,b)$ be obtained from $\chi_0(a,b)$ by recoloring exactly
one edge between $A$ and $B$ with the color used within the two parts. 
Different choices of this cross-edge give equivalent colorings. For $a,b\geq2$, exactly one of the two color graphs is connected
in $\chi_0(a,b)$, whereas both color graphs are connected
in $\chi_1(a,b)$. This property is preserved under vertex relabeling and exchanging the colors, so the two colorings are not equivalent.

Let $k\geq5$ be odd. The coloring $\chi_0(k-1,k-1)$ is a coloring of
$K_{r(C_k)-1}=K_{2k-2}$ containing no monochromatic copy of $C_k$.
On the other hand, both $\chi_0(k,k-1)$ and $\chi_1(k,k-1)$ are
colorings of $K_{r(C_k)}=K_{2k-1}$ containing exactly
\[
\frac{k!}{2k}=\frac{(k-1)!}{2}
\]
monochromatic copies of $C_k$. In each coloring, every monochromatic
copy of $C_k$ is a Hamilton cycle in the monochromatic clique on the
$k$ vertices of $A$. 
More generally, for every odd integer $k$ and all positive integers
$a$ and $b$, each of $\chi_0(a,b)$ and $\chi_1(a,b)$ contains exactly
\[
\left(\binom{a}{k}+\binom{b}{k}\right)\frac{(k-1)!}{2}
\]
monochromatic copies of $C_k$.

For $\chi_1(k,k-1)$, the exceptional cross-edge
is a bridge in its color and therefore lies in no monochromatic cycle.  
Consequently,
\begin{equation}\label{eq:upper-construction}
m(C_k)\leq \frac{(k-1)!}{2}
\end{equation}
 for every odd integer $k\geq 5$.
Motivated by this construction, Conlon, Fox, Sudakov and Wei~\cite{ConlonFoxSudakovWei,ConlonFoxSudakovWeiPaths} formulated
the following conjecture.

\begin{conjecture}[Conlon, Fox, Sudakov, and Wei
\cite{ConlonFoxSudakovWei,ConlonFoxSudakovWeiPaths}]\label{conj:CFSW-odd}
For every sufficiently large odd integer $k$,
\[
m(C_k)=\frac{(k-1)!}{2}.
\]
\end{conjecture}

By combining the regularity lemma with a stability argument, Conlon,
Fox, Sudakov and Wei~\cite{ConlonFoxSudakovWei} show that either there
is an odd cyclic sequence of clusters in the regularity partition such
that every consecutive pair is regular and has density bounded below
by a fixed positive constant in the same color, or the original
coloring is close to the split construction.
 In the first case, they
count cycles by following the cluster sequence; in the second, they
count cycles arising from local defects in the split structure. This
gives the lower bound $m(C_k)\geq(ck)^k$.

Their argument determines the correct factorial order up to an
exponential factor, but its counting estimates are not sufficiently
precise to establish the exact value in
Conjecture \ref{conj:CFSW-odd}. Indeed, retaining only a fixed proportion of the available
choices at each of $k$ steps causes an exponential loss. This loss can be absorbed into the constant $c$ in $(ck)^k$,
but a lower bound of this form alone does not establish the
precise value $(k-1)!/2$. The estimates in the near-split case
are likewise too coarse to determine the equality cases.


We follow the same broad division into near-split and
far-from-split colorings, but refine the counting in both cases, thereby retaining the information lost in the earlier estimates.

Our result determines the multiplicity not only at the Ramsey threshold, 
but also at every fixed number of vertices above it, 
and determines all colorings attaining equality.

\begin{theorem}\label{thm:main}
Fix a nonnegative integer $\ell$. There is an integer $k_0=k_0(\ell)$ such that the following holds for every odd integer $k\ge k_0$. 
Every red/blue edge-coloring $\chi$ of $K_{2k-1+\ell}$ satisfies
\[
\Nmono(C_k,\chi)\ge\left(\binom{\alpha}{k}+\binom {\beta}{k}\right)\frac{(k-1)!}{2},
\]
where $\alpha=\lceil (2k-1+\ell)/2\rceil$ and $\beta=\lfloor (2k-1+\ell)/2\rfloor$.
Equality holds if and only if $\chi$ is equivalent to $\chi_0(\alpha,\beta)$ or $\chi_1(\alpha,\beta)$. Consequently,
\[
M(C_k,2k-1+\ell)=\left(\binom{\alpha}{k}+\binom {\beta}{k}\right)\frac{(k-1)!}{2}.
\]
\end{theorem}

The case $\ell=0$ confirms Conjecture~\ref{conj:CFSW-odd} and classifies all threshold extremal colorings.


At the end of this section, we give an overview of the proof. We divide the colorings according to
whether they are close to a split coloring. In the near-split case, a
cleaning argument produces two large sets that almost have the split
pattern. 
Falling-factorial estimates show that a wrong-colored internal edge, two vertex-disjoint short paths in the internal color joining the parts, or an external vertex with neighbors in both parts in the color of the cut produces more cycles than the bound in Theorem~\ref{thm:main}. 
Otherwise, all vertices can be divided into two
monochromatic cliques. A balancing argument gives the required lower
bound; equality forces the two parts to be as equal as possible and
allows at most one cross-edge in the clique color.

In the far-from-split case, 
we first apply the colored
regularity lemma in the form used in~\cite{ConlonFoxSudakovWei}. This
divides the vertices into a bounded number of almost equal clusters and
records the red and blue densities between every regular pair.
That is, the colored regularity lemma replaces the coloring by a bounded weighted graph whose vertices represent clusters.
A structural lemma finds, in one color, a connected collection of
clusters containing an odd cycle and supporting many closed routes
without using any cluster too often. We round the required numbers of
moves between clusters to form a directed multigraph with exactly $k$
arcs and equal in-degree and out-degree at every vertex. The BEST
theorem counts its Euler tours, and regularity turns the resulting
cluster orders into simple monochromatic copies of $C_k$. This gives an
exponential gain over $(k-1)!/2$, whereas the bound in
Theorem~\ref{thm:main} is at most a polynomial factor larger because
$\ell$ is fixed. Hence the far-from-split case always gives a strict
inequality.

\section{Cycle-counting estimates}\label{Cycle-counting estimates}
Throughout the paper, $\ell$ is a fixed nonnegative integer, and we write
$$k=2m+1,~N=2k-1+\ell,~H_m=\frac{(2m)!}{2}=\frac{(k-1)!}{2},$$ 
and 
$$F_\ell(k)=\left(\binom{\lceil N/2\rceil}{k}+\binom{\lfloor N/2\rfloor}{k}\right)H_m.$$
Clearly, $F_0(k)=H_m$. The following standard balancing inequality, which is a
two-variable consequence of the Hardy--Littlewood--Pólya
majorization principle \cite{HardyLittlewoodPolya1929}, shows that
$F_\ell(k)$ is the minimum number of copies of $C_k$ lying entirely
within two disjoint monochromatic cliques of the same color whose
orders sum to $N$.
\begin{lemma}\label{lem:balance}
Let $k\ge2$ and let $p,q$ be nonnegative integers with $p+q=N\geq 2k-1$. 
If $\alpha=\lceil N/2\rceil$ and $\beta=\lfloor N/2\rfloor$, 
then $\binom pk+\binom qk\ge\binom{\alpha}{k}+\binom {\beta}{k}$. 
Equality holds if and only if $\{p,q\}=\{\alpha,\beta\}$.
\end{lemma}

For a graph $G$ and $U\subseteq V(G)$, let $G[U]$ denote the
subgraph of $G$ induced by $U$. All logarithms are natural.
For a real number $x$ and a nonnegative integer $s$, write $(x)_s=x(x-1)\cdots(x-s+1)$, with $(x)_0=1$. Stirling's formula is
\[
\log(n!) = n\log n - n + \frac{1}{2}\log(2\pi n) + O\left(\frac{1}{n}\right).
\]

We have the following estimate for $F_\ell(k)$.

\begin{lemma}\label{lem:estimate-Fl(k)}
There is a constant \(C_\ell>0\), depending only on \(\ell\),
such that
\begin{equation*}
  F_\ell(k)\le C_\ell m^\ell H_m
\end{equation*}
for all sufficiently large \(k\).
\end{lemma}

\begin{proof}
Since \(N=2k-1+\ell\), both \(\lceil N/2\rceil\) and
\(\lfloor N/2\rfloor\) are at most \(k+\ell\).
By monotonicity in the upper argument,
\[
  \frac{F_\ell(k)}{H_m}
  \le 2\binom{k+\ell}{k}
  =2\binom{k+\ell}{\ell}
  \le \frac{2(k+\ell)^\ell}{\ell!}.
\]
For \(m\ge\ell+1\), we have \(k+\ell=2m+1+\ell\le3m\).
Consequently, the lemma holds with
\(C_\ell=2\cdot3^\ell/\ell!\).
\end{proof}

The cycle constructions used below produce products of two falling
factorials. In each product, the upper arguments are close to $2m$,
the lengths are close to $m$, and the loss caused by prescribed
endpoints or previously used vertices is bounded by a constant. The
next lemma shows that every such product exceeds $H_m$ times any fixed
polynomial in $m$. In view of Lemma~\ref{lem:estimate-Fl(k)}, taking
$C=C_\ell$ and $d=\ell$ will therefore allow us to compare all these
cycle counts directly with $F_\ell(k)$.
\begin{lemma}\label{lem:falling-comparison}
There is an absolute constant $\zeta_0>0$ with the following property.
Fix nonnegative integers $a_1,a_2,b_1,b_2,d$ and a positive constant $C$. If $m$ is sufficiently
large and $s,t\geq(2-\zeta_0)m$, then
\[
(s-a_1)_{m-b_1}(t-a_2)_{m-b_2}>Cm^d H_m.
\]
\end{lemma}

\begin{proof}
Fix $0\leq\zeta<1$. Suppose that $x_m$ and $q_m$ are integers satisfying
$x_m=(2-\zeta)m+O(1)$ and $q_m=m+O(1)$. Then
$x_m-q_m=(1-\zeta)m+O(1)$. Since
$(x_m)_{q_m}=x_m!/(x_m-q_m)!$,
Stirling's formula gives
\begin{align*}
\log (x_m)_{q_m}
&=\log(x_m)!-\log(x_m-q_m)!\\
&=x_m\log x_m-(x_m-q_m)\log(x_m-q_m)-q_m
  +O(\log m)\\
&=m\log m+m\Phi(\zeta)+O(\log m),
\end{align*}
where $\Phi(\zeta)
=(2-\zeta)\log(2-\zeta)
 -(1-\zeta)\log(1-\zeta)-1$.
On the other hand, since $H_m=(2m)!/2$, Stirling's formula gives
\begin{align*}
\log H_m
&=\log(2m)!-\log2\\
&=2m\log(2m)-2m+O(\log m)\\
&=2m\log m+(2\log2-2)m+O(\log m).
\end{align*}
At $\zeta=0$, we have $\Phi(0)=2\log2-1$
and hence $2\Phi(0)-(2\log2-2)=2\log2>0$.
By continuity, we may choose an absolute constant
$\zeta_0\in(0,1)$ such that $2\Phi(\zeta_0)-2\log2+2>0$.

For fixed $a_1,a_2,b_1,b_2$ and all sufficiently large $m$, every
factor in the falling factorials below is positive. Since $(x)_r$ is
increasing in $x$ whenever all its factors are positive, the assumptions
$s,t\geq(2-\zeta_0)m$ imply
\[
(s-a_1)_{m-b_1}(t-a_2)_{m-b_2}
\geq
\bigl(\lfloor(2-\zeta_0)m\rfloor-a_1\bigr)_{m-b_1}
\bigl(\lfloor(2-\zeta_0)m\rfloor-a_2\bigr)_{m-b_2}.
\]
Applying the preceding asymptotic estimate with
$x_m=\lfloor(2-\zeta_0)m\rfloor-a_i$ and $q_m=m-b_i$ for $i\in\{1,2\}$ gives
\begin{align*}
\log\bigl((s-a_1)_{m-b_1}(t-a_2)_{m-b_2}\bigr)-\log H_m
&\geq
\left(2\Phi(\zeta_0)-2\log2+2\right)m+O(\log m)\\
&>\log (Cm^d)
\end{align*}
for sufficiently large $m$. 
Exponentiating proves $(s-a_1)_{m-b_1}(t-a_2)_{m-b_2}>Cm^d H_m$.
\end{proof}
All cycles counted in this section are unrooted and unoriented.
The following three lemmas correspond to the three mechanisms in the near-split proof, where a single local defect immediately creates many odd cycles.

We first consider a wrong-colored edge inside one of the two parts.
If the other part supplies many common neighbors for pairs of vertices,
then this edge can be extended to many odd cycles. The
following lemma gives the required count.
\begin{lemma}[Conlon, Fox, Sudakov and Wei
\cite{ConlonFoxSudakovWei}]\label{lem:wrong-edge}
Let $S$ and $T$ be disjoint vertex sets in a graph of one fixed color. Suppose that $S$ contains an edge $uu'$ and every two vertices of $S$ have at least $s$ common neighbors in $T$, where $s\geq m$. Then there are at least
\[
(|S|-2)_{m-1}(s)_m
\]
cycles $C_{2m+1}$ of that color containing $uu'$.
\end{lemma}


After wrong-colored internal edges have been eliminated, the two main parts become monochromatic cliques. Two vertex-disjoint short paths of the same color between these cliques can then be joined by paths inside the cliques to form an odd cycle. The next lemma counts the cycles obtained in this way.
\begin{lemma}[Conlon, Fox, Sudakov and Wei
\cite{ConlonFoxSudakovWei}]\label{lem:two-connectors}
Let $m\geq3$, and let $S$ and $T$ be disjoint cliques in one fixed color with $|S|,|T|\geq m+1$. Suppose that there are two vertex-disjoint paths $P_1,P_2$ of that color, each of length one or two, with one end in $S$, the other in $T$, and all internal vertices outside $S\cup T$. Then that color contains at least
\begin{equation*}
(|S|-2)_{m-3}(|T|-2)_{m-2}
\end{equation*}
cycles $C_{2m+1}$.
\end{lemma}


A vertex outside the two cleaned parts may also create many cycles. If
it forms a monochromatic path of length two between the parts, then this
path can be completed by an alternating path across the monochromatic
cut. The following lemma records the resulting count.
\begin{lemma}\label{lem:external-two-path}
Let every edge between disjoint sets $S$ and $T$ have one fixed color. If $a-z-b$ is a path of that color with $a\in S$, $b\in T$ and $z\notin S\cup T$, then at least
\[
(|S|-1)_{m-1}(|T|-1)_{m-1}
\]
cycles $C_{2m+1}$ of that color contain this path.
\end{lemma}

\begin{proof}
If $|S|<m$ or $|T|<m$, the asserted lower bound is zero and the result is immediate. We may therefore assume that $|S|,|T|\geq m$. Choose an alternating simple path of length $2m-1$ from $b$ to $a$ in the complete bipartite graph between $S$ and $T$. Apart from its ends, this requires ordered choices of $m-1$ vertices from each of $S\setminus\{a\}$ and $T\setminus\{b\}$. Adding $a-z-b$ closes the alternating path into a cycle. Since the above construction yields distinct cycles, there are at least \((|S|-1)_{m-1}(|T|-1)_{m-1}\) cycles.
\end{proof}

\section{Colorings close to the split construction}

For a color $\xi\in\{R,B\}$ and disjoint
nonempty vertex sets $X$ and $Y$, let $e_\xi(X,Y)$ denote the number
of $\xi$-colored edges joining $X$ and $Y$. The density of one color between $X$ and $Y$ is defined by
\[
d_\xi(X,Y):=\frac{e_\xi(X,Y)}{|X||Y|}.
\]
If $|X|\geq2$, let $e_\xi(X)$ denote the number of $\xi$-colored edges
with both endpoints in $X$, and the density inside $X$ is defined by
\[
d_\xi(X):=\frac{e_\xi(X)}{\binom{|X|}{2}}.
\]
Thus
$d_{R}(X,Y)+d_{B}(X,Y)=1$ and
$d_{R}(X)+d_{B}(X)=1$.

We also use colored neighborhoods and degrees. For a color $\xi\in\{R,B\}$, a vertex $v$, and a vertex set $U$, let $N_\xi(v,U)$ be the set of vertices $u\in U\setminus\{v\}$ for which $uv$ has color $\xi$, and write $d_\xi(v,U)=|N_\xi(v,U)|$.


\begin{definition}[$\lambda$-near-split coloring]
A red/blue edge-coloring of $K_N$ is called \emph{$\lambda$-near-split} if there exists a partition $V(K_N)=A\cup B$ with $|A|,|B|\ge (1/2-\lambda)N$ such that, after possibly exchanging the two colors, the red densities inside both $A$ and $B$, and the blue density between $A$ and $B$, are all at least $1-\lambda$.
\end{definition}


We now show that a near-split coloring with at most the target number of cycles has a partition of its entire vertex set into two cliques of one color.
\begin{prop}\label{prop:near-split}
Fix a nonnegative integer $\ell$. 
There is an absolute constant $\lambda_*>0$ with the following property. 
For every sufficiently large odd integer $k$, 
if a $\lambda_*$-near-split coloring $\chi$ of $K_N$ satisfies $\Nmono(C_k,\chi)\le F_\ell(k)$, then $V(K_N)$ has a partition into two monochromatic cliques of the same color.
\end{prop}

\begin{proof}
Let \(\zeta_0\) be the constant supplied by
Lemma~\ref{lem:falling-comparison}. Choose
\(\eta_0\in(0,\zeta_0/100)\), and then choose an absolute constant
\(\lambda_*>0\) so small that \(\lambda_*<\eta_0^2/4\).
Finally, choose $k_0$ sufficiently large in terms of $\ell$ so that
all estimates below hold whenever $k\ge k_0$.

Let \(\chi\) be a \(\lambda_*\)-near-split coloring of \(K_{N}\).
Choose a partition \(V(K_{N})=A\mathbin{\cup}B\) satisfying this
property. After exchanging the colors if necessary, assume that the red
densities inside \(A\) and \(B\), and the blue density between \(A\) and
\(B\), are all at least \(1-\lambda_*\). Let 
\[
\eta=\sqrt{\lambda_*+\frac{1}{k}}.
\]
Then \(\eta<\eta_0\).

We first remove vertices whose degrees do not reflect the three density assumptions. 
Set 
$X_1=\{a \in A:d_R(a,A)<(1-\eta)|A|\}$ and $X_2=\{a \in A:d_B(a,B)<(1-\eta)|B|\}$.
Write \(x=|X_1|/|A|\). Since the
vertices in \(X_1\) have red degree less than \((1-\eta)|A|\), while
every other vertex has red degree less than \(|A|\), the average red
degree in \(A\) is less than $x(1-\eta)|A|+(1-x)|A|=(1-x\eta)|A|$. On the other hand,
the red density inside \(A\) is at least \(1-\lambda_*\), so the average red
degree in \(A\) is at least \((1-\lambda_*)(|A|-1)\). Consequently,
\[
x\eta
<\lambda_*+\frac{1-\lambda_*}{|A|}
\leq\lambda_*+\frac1{|A|}.
\]
Since \(|A|\geq(1/2-\lambda_*)(2k-1+\ell)\), the choices of \(\lambda_*\)
and \(k\) ensure that \(|A|\geq k/2\). Using
\(\eta^2=\lambda_*+1/k\), we obtain
$x\eta\leq\lambda_*+2/k\leq2\eta^2$.
Therefore \(x\leq2\eta\), and hence \(|X_1|\leq2\eta|A|\).
The same averaging argument gives
\(|X_2|\leq2\eta|A|\), and we define
\(A'=A\setminus(X_1\cup X_2)\).
 Then \(|A'|\geq(1-4\eta)|A|\).

 Set $Y_1=\{b\in B:d_R(b,B)<(1-\eta)|B|\}$ and $Y_2=\{b\in B:d_B(b,A)<(1-\eta)|A|\}$.
Let \(B'=B\setminus(Y_1\cup Y_2)\). Similarly, we have
\(|Y_1|,|Y_2|\leq2\eta|B|\) and \(|B'|\geq(1-4\eta)|B|\).

Moreover, the $\lambda_*$-near-split assumption and
\(\eta^2=\lambda_*+1/k\) imply, for sufficiently large \(k\),
\[
|A|,|B|\geq\left(\frac12-\lambda_*\right)(2k-1+\ell)
\geq(1-2\eta)k.
\]
Consequently,
\begin{equation}\label{eq:clean-size}
|A'|,|B'|\geq(1-4\eta)(1-2\eta)k
\geq(1-8\eta)k.
\end{equation}
Every $a\in A'$ satisfies $d_R(a,A')\geq(1-5\eta)|A|$ and $d_B(a,B')\geq(1-5\eta)|B|$
and the analogous inequalities hold for every vertex of $B'$. In particular, any two vertices of $A'$ have at least
\begin{equation}\label{eq:common-blue}
2(1-5\eta)|B|-|B'|\geq(1-14\eta)k
\end{equation}
common blue neighbors in $B'$. The same statement holds after exchanging $A'$ and $B'$.

Since \(\eta<\eta_0<\zeta_0/100\) and
\(k=2m+1\), the choice of \(k_0\) ensures that $(1-14\eta)k-O(1)\geq(2-\zeta_0)m$.
Hence Lemma~\ref{lem:falling-comparison} applies whenever it is invoked
below.
If $A'$ contains a blue edge, Lemma \ref{lem:wrong-edge}, \eqref{eq:clean-size}, \eqref{eq:common-blue}, and Lemma \ref{lem:falling-comparison} produce at least $(|A'|-2)_{m-1}((1-14\eta)k)_m>F_{\ell}(k)$ blue copies of $C_k$, a contradiction. 
The same applies to $B'$. We may consequently assume that $A'$ and $B'$ are red cliques.

There cannot be two vertex-disjoint red paths of length at most two joining $A'$ to $B'$, 
unless there are already more than $F_{\ell}(k)$ red copies of $C_k$. To see this, replace a two-edge path whose internal vertex belongs to $A'\cup B'$ by the red cross-edge contained in that path. The two resulting paths remain vertex-disjoint and have all internal vertices outside $A'\cup B'$. Lemmas \ref{lem:two-connectors} and \ref{lem:falling-comparison} now apply. In particular, the red bipartite graph between $A'$ and $B'$ has matching number at most one. Its edges therefore have a common end. Deleting that end when necessary leaves red cliques $A_0\subseteq A'$ and $B_0\subseteq B'$ with every edge between them blue.

Let $Z=V(K_N)\setminus(A_0\cup B_0)$. 
If some $z\in Z$ has a blue neighbor in each of $A_0$ and $B_0$, 
Lemmas \ref{lem:external-two-path} and \ref{lem:falling-comparison} again yield more than $F_{\ell}(k)$ blue cycles. 
We may thus partition $Z=Z_1\cup Z_2$ so that each vertex of $Z_1$ is red to all of $A_0$ and each vertex of $Z_2$ is red to all of $B_0$. A vertex that is red to both sides may be assigned arbitrarily.

Suppose that there are red edges $z_1b_1$ and $z_2a_2$ with $z_1\in Z_1$, $b_1\in B_0$, $z_2\in Z_2$, and $a_2\in A_0$. Choose $a_1\in A_0\setminus\{a_2\}$ and $b_2\in B_0\setminus\{b_1\}$. 
Then the paths $a_1-z_1-b_1$ and $a_2-z_2-b_2$ are vertex-disjoint red paths
of length two, so Lemmas \ref{lem:two-connectors} and
\ref{lem:falling-comparison} imply that there are more than
$F_{\ell}(k)$ red copies of $C_k$. After exchanging the two sides, we may assume that every edge between $Z_1$ and $B_0$ is blue.

The red bipartite graph between $A_0$ and $Z_2$ also has matching number at most one. Otherwise, two disjoint red edges $a_1z_1$ and $a_2z_2$, together with two distinct vertices of $B_0$, form two vertex-disjoint red paths of length two joining $A_0$ to $B_0$. Delete the common end of all red edges between $A_0$ and $Z_2$, if such an end is needed. This produces $A_1\subseteq A_0$ and $Z'_2\subseteq Z_2$ such that every edge between $A_1$ and $Z'_2$ is blue. Only this last deletion permanently removes a vertex. The earlier deletion from $A'\cup B'$ placed its vertex in $Z$.

Define $\widetilde A=A_1\cup Z_1$ and $\widetilde B=B_0\cup Z'_2$.
These sets are disjoint and satisfy
\begin{equation}\label{eq:two-red-sets-sum}
|\widetilde A|+|\widetilde B|\geq N-1=2k-2+\ell.
\end{equation}
Every vertex of $\widetilde A$ is blue to all of $B_0$, and every vertex of $\widetilde B$ is blue to all of $A_1$. Both $A_1$ and $B_0$ have order $(1-O(\eta))k$. If $\widetilde A$ contains a blue edge, its ends have all of $B_0$ as common blue neighbors. 
Lemmas \ref{lem:wrong-edge} and \ref{lem:falling-comparison} then give more than $F_{\ell}(k)$ blue cycles. The same argument applies to $\widetilde B$. We may therefore assume that both $\widetilde A$ and $\widetilde B$ are red cliques.

If their union is $V(K_N)$, the proof is complete. Otherwise, let $x$ be the unique remaining vertex. Suppose $x$ has blue neighbors $u\in\widetilde A$ and $v\in\widetilde B$. Choose distinct $b_1,\ldots,b_{m-1}\in B_0\setminus\{v\}$ and distinct $a_1,\ldots,a_{m-1}\in A_1\setminus\{u\}$. Then
\[
x,u,b_1,a_1,b_2,a_2,\ldots,b_{m-1},a_{m-1},v,x
\]
is a blue cycle of length $2m+1$. Following it from $x$ through the fixed neighbor $u$ recovers both ordered lists. Distinct choices therefore yield distinct unrooted cycles, and their number is at least $(|B_0|-1)_{m-1}(|A_1|-1)_{m-1}>F_\ell(k)$, a contradiction. Thus $x$ has no blue neighbor in one of the two cliques. It is red to every vertex of that clique and can be added to it. The resulting two red cliques partition $V(K_N)$.
\end{proof}

\section{Weighted templates and a robust non-bipartite core}


We next develop the main structural tool for colorings that are far from every split coloring. 



Let $I$ be a finite index set, and let $c=(c_i)_{i\in I}$ satisfy $c_i>0$ for every $i\in I$ and $\sum_{i\in I}c_i=2$.
In the later regularity application, $I$ indexes the clusters $V_i$,
and $c_i=2|V_i|/N$ is the normalized size of $V_i$. 

Let $p=(p_{ij})$ and $q=(q_{ij})$ be symmetric matrices. 
Their intended
meaning is that $p_{ij}$ and $q_{ij}$ record the retained red and blue
densities, respectively, between clusters $V_i$ and $V_j$. Densities
below a fixed threshold, as well as densities on irregular pairs, are
discarded and replaced by zero. We therefore assume $p_{ii}=q_{ii}=0$ and $0\leq p_{ij},q_{ij}\leq1$ with $p_{ij}+q_{ij}\leq1$. 
Define
\[
R_p=\sum_{i,j\in I}c_ic_jp_{ij}
\quad \text{and} \quad 
R_q=\sum_{i,j\in I}c_ic_jq_{ij}.
\]

For any specified collection of ordered pairs, its ordered capacity, $p$-mass, and $p$-defect are obtained by summing $c_ic_j$, $c_ic_jp_{ij}$, and $c_ic_j(1-p_{ij})$, respectively, over those pairs. The $q$-mass and $q$-defect are defined analogously. Relative to a partition of $I$, a defect is called internal or cross according as the corresponding ordered pairs lie within the parts or between them. Define
\begin{equation}\label{eq:total-defect}
\Delta(p,q)=\sum_{i,j\in I}c_ic_j(1-p_{ij}-q_{ij}).
\end{equation}
The difference $1-p_{ij}-q_{ij}$ is the density omitted from the pair
$(i,j)$.
Thus $\Delta(p,q)$ is the total weighted density omitted from the
template.

For $X\subseteq I$, write $c(X)=\sum_{i\in X}c_i$ for its capacity
and define
\begin{align}\label{eq:split-defect}
D_{p,q}(X)={}&|c(X)-1|
 +\sum_{\substack{i,j\in X\\\text{or }i,j\notin X}}c_ic_j(1-p_{ij}) +\sum_{\substack{i\in X,\,j\notin X\\\text{or }i\notin X,\,j\in X}}c_ic_j(1-q_{ij}).
\end{align}
Thus $D_{p,q}(X)$ is $|c(X)-1|$ plus the internal $p$-defect and the cross $q$-defect for the partition with parts $X$ and $I\setminus X$.
The three terms measure the imbalance of the two parts, the missing
red density within the parts, and the missing blue density across the
cut. Hence $D_{p,q}(X)$ is small when this partition is close to a
split coloring with red inside the parts and blue between them.

The distance to a balanced split template is
\[
\spl(p,q)=\min_{X\subseteq I}\min\{D_{p,q}(X),D_{q,p}(X)\}.
\]
The first minimum chooses the partition, and the second allows the two
colors to be exchanged. Thus $\spl(p,q)$ measures how close the
weighted template is to any balanced split coloring.

For a symmetric nonnegative matrix $w=(w_{ij})$, let $G_w$ be the graph on $I$ in which $ij$ is an edge precisely when $i\neq j$ and $w_{ij}>0$. As usual, we identify each connected component of $G_w$ with its vertex set. For every nonempty subset $J\subseteq I$, write $\mu_w(J)=\sum_{i,j\in J}c_ic_jw_{ij}$.
For each $i\in J$, define
$d_{w,J}(i)=\sum_{j\in J}c_jw_{ij}$.
Thus $\mu_w(J)=\sum_{i\in J}c_i d_{w,J}(i)$.

The weighted counting argument used later leads to an entropy expression in the normalized degrees $d_{w,J}(i)/\mu_w(J)$. We first isolate its one-variable part and record the lower bound needed below.

Let
\[
f(z)=z\log z-(1-z)\log(1-z),
\qquad 0\leq z\leq1.
\]
Here and below, we interpret $0\log0$ as $0$.

\begin{lemma}\label{lem:f-bound}
For every $0\leq z\leq1$,
\begin{equation}\label{eq:f-bound}
f(z)\geq-\frac{1-z}{2}.
\end{equation}
Consequently, if $G_w[J]$ is connected and non-bipartite and $d_{w,J}(i)\leq\mu_w(J)$ for every $i\in J$, then
\begin{equation}\label{eq:gamma-lower}
\Gamma(w,c;J):=\log\mu_w(J)+\sum_{i\in J}c_i f\left(\frac{d_{w,J}(i)}{\mu_w(J)}\right)
\geq \log\mu_w(J)-\frac{c(J)-1}{2}.
\end{equation}
\end{lemma}

\begin{proof}
Since $f''(z)=(1-2z)/(z(1-z))\leq0$ for $1/2<z<1$ and $f(1/2)=f(1)=0$, we have $f(z)\geq0$ for $1/2\leq z\leq1$. If $0\leq z\leq1/4$, then
\[
f(z)+\frac{1-z}{2}\geq z\log z+\frac{1-z}{2},
\]
and the expression on the right is decreasing on this interval and positive at $z=1/4$. It follows that 
\[f(z)+\frac{1-z}{2}\geq \frac{3}{8}-\frac{1}{4}\log 4>0.\]
If $1/4\leq z\leq1/2$, then $f'(z)=\log(z(1-z))+2>0$, so $f(z)\geq f(1/4)>-1/4\geq-(1-z)/2$. This proves \eqref{eq:f-bound}. For the second assertion, abbreviate $\mu=\mu_w(J)$ and $d_i=d_{w,J}(i)$. Since $\sum_{i\in J}c_id_i/\mu=1$, summing \eqref{eq:f-bound} over $i\in J$ yields \eqref{eq:gamma-lower}.
\end{proof}

\begin{lemma}\label{lem:template-near-equality}
Suppose that $R_p\geq2-\Delta/2$ and $\Delta=\Delta(p,q)\leq 1/100$.
If every non-bipartite component of $G_p$ has capacity at most one, then
\begin{equation}\label{eq:near-equality-split}
\spl(p,q)\leq40\sqrt{\Delta}.
\end{equation}
\end{lemma}

\begin{proof}
Let $J_1,\ldots,J_t$ be the components of $G_p$, and write $C_s=c(J_s)$ and $r_s=\mu_p(J_s)$ for $1\leq s\leq t$. For every $s$, we have $r_s\leq C_s$. Indeed, if $J_s$ is non-bipartite, then $C_s\leq1$ by assumption, and hence $r_s\leq C_s^2\leq C_s$. If $J_s$ is bipartite with parts $U_s$ and $V_s$, then
$r_s\leq2c(U_s)c(V_s)\leq C_s^2/2\leq C_s$, where the last inequality follows from $C_s\leq2$. The components $J_1,\ldots,J_t$ form a partition of $I$, so $\sum_{s=1}^t C_s=2$. Moreover, $p_{ij}=0$ whenever $i\in J_s$ and $j\in J_{s'}$ for distinct $s,s'$, and hence $\sum_{s=1}^t r_s=R_p$. Therefore
\begin{equation}\label{eq:rho-deficit}
\rho:=2-R_p=\sum_s(C_s-r_s),
\quad
0\leq\rho\leq\frac{\Delta}{2}.
\end{equation}

First, suppose that $J$ is a bipartite component of $G_p$ with bipartition $(U,V)$ and capacity $C>1$. Write $a=c(U)$, $b=c(V)$, and $\omega=2-C$. If $r=\mu_p(J)$, then
\[
C-r\geq C-2ab=\frac{C\omega}{2}+\frac{(a-b)^2}{2}.
\]
Together with \eqref{eq:rho-deficit}, this yields
\[
\omega\leq2\rho,
\quad |a-b|\leq\sqrt{2\rho},
\quad 2ab-r\leq\rho.
\]
Take $X=U$ and compare with the split template having $q$ inside the parts and $p$ across them. We have
\[
|c(X)-1|\leq \frac{\omega}{2}+\frac{|a-b|}{2}\leq\rho+\sqrt{\frac\rho2}.
\]
The $q$-defect inside $U$ and $V$ is at most $\Delta$. All ordered pairs having an end outside $J$ have total capacity at most $4\omega$. The $p$-defect across $U$ and $V$ is $2ab-r\leq\rho$, and the remaining cross $p$-defect is at most $4\omega$. Therefore
\[
D_{q,p}(X)\leq2\rho+\sqrt{\frac\rho2}+\Delta+8\omega\leq40\sqrt\Delta.
\]

We may now assume that every bipartite component has capacity at most one. For every bipartite component $J_s$, we have
$C_s-r_s\geq C_s-C_s^2/2\geq C_s/2$. The same bound holds for every non-bipartite component $J_s$ with $C_s\leq1/2$, since
$C_s-r_s\geq C_s-C_s^2\geq C_s/2$. By \eqref{eq:rho-deficit}, the components in these two classes have total capacity at most $2\rho$.
Every component outside these two classes is non-bipartite and satisfies $1/2<C_s\leq1$ and
\[
C_s-r_s\geq C_s-C_s^2\geq\frac{1-C_s}{2},
\]
so $C_s\geq1-2\rho$. Since the capacities of all components sum to $2$, these components have total capacity at least $2-2\rho$. As $2-2\rho>1$ and each such component has capacity at most $1$, there are at least two such components. Since each has capacity at least $1-2\rho>3/4$, there are at most two. Hence there are exactly two such components, which we denote by $J_1$ and $J_2$. Let $O$ be the union of the vertex sets of all other components. Then $c(O)\leq2\rho$.

Take $X=J_1$ and compare with the split template having $p$ inside the parts and $q$ across them. We have $|c(X)-1|\leq2\rho$. For $s\in\{1,2\}$, the internal $p$-defect on $J_s$ is $C_s^2-r_s\leq C_s-r_s$, since $C_s\leq1$. Hence the total internal $p$-defect on $J_1$ and $J_2$ is at most $\rho$. The ordered capacity of pairs with at least one endpoint in $O$ is at most $4c(O)\leq8\rho$. Thus the total internal $p$-defect is at most $9\rho$. There are no positive $p$-weights between $J_1$ and $J_2$, so the cross $q$-defect between them is at most $\Delta$. The ordered capacity of the cross-pairs between $J_1$ and $O$ is at most $2C_1c(O)\leq4\rho$. Consequently,
\[
D_{p,q}(X)\leq2\rho+9\rho+\Delta+4\rho
=15\rho+\Delta
\leq40\sqrt{\Delta}.
\]
This proves \eqref{eq:near-equality-split}.
\end{proof}

\begin{lemma}\label{lem:robust-core}
For every $\delta>0$, there are positive constants $\eta,\gamma,\sigma$, and $a_0$ with the following property. Suppose that $\Delta(p,q)\leq\eta$ and $\spl(p,q)\geq\delta$.
Then there exist a set $J\subseteq I$ and a symmetric nonnegative matrix $w=(w_{ij})_{i,j\in I}$. The matrix $w$ vanishes whenever $i\notin J$ or $j\notin J$, and $G_w[J]$ is connected and non-bipartite. Moreover, either $w_{ij}\leq p_{ij}$ for all $i,j\in I$, or $w_{ij}\leq q_{ij}$ for all $i,j\in I$. It also satisfies
\begin{equation}\label{eq:capacity-margin}
d_{w,J}(i)\leq(1-\sigma)\mu_w(J)
\quad\text{for every }i\in J,
\end{equation}
and $\Gamma(w,c;J)\geq\gamma$.
If every positive entry of $p$ and $q$ is at least $\tau>0$, then $w$ may be chosen so that every positive entry of $w$ is at least $a_0\tau$.
\end{lemma}

\begin{proof}
Whenever a matrix is restricted to a subset of $I$, we extend the restriction by zero outside that subset.
Since $R_p+R_q=4-\Delta(p,q)$, an exchange of $p$ and $q$ permits us to assume
\begin{equation}\label{eq:global-mass-large}
R_p\geq2-\frac{\Delta(p,q)}2.
\end{equation}
Choose $\eta\leq1/100$ so small that $40\sqrt\eta<\delta$. Lemma \ref{lem:template-near-equality} shows that $G_p$ has a connected non-bipartite component $J$ of capacity greater than one. Write
\[
C=c(J)=1+s,
\quad
\omega=2-C=1-s,
\quad \text{and} \quad
\mu=\mu_p(J).
\]
There are no positive $p$-weights between $J$ and its complement, whose total $p$-mass is at most $\omega^2$. It follows from \eqref{eq:global-mass-large} that
\begin{equation}\label{eq:component-mass-lower}
\mu\geq2-\omega^2-\frac{\Delta(p,q)}2
=C+s\omega-\frac{\Delta(p,q)}2.
\end{equation}

We first record that $s$ is bounded away from zero in terms of $\delta$. Compare the template with the split whose two parts are $J$ and $I\setminus J$, with $p$ as the internal color. The internal $p$-defect is at most
\[
C^2+\omega^2-R_p\leq2s^2+\frac{\Delta(p,q)}2,
\]
and the cross $q$-defect is at most $\Delta(p,q)$. Hence
\begin{equation}\label{eq:s-lower-preparation}
\spl(p,q)\leq s+2s^2+\frac{3\Delta(p,q)}2.
\end{equation}
Choose $s_0=s_0(\delta)>0$ so that $s_0+2s_0^2<\delta/2$, and decrease $\eta$ so that $3\eta/2<\delta/2$. Equation \eqref{eq:s-lower-preparation} then implies
\begin{equation}\label{eq:s-positive}
s\geq s_0.
\end{equation}

Choose a constant $\theta>0$ such that
\[
13\theta<\frac1{16}
\quad \text{and} \quad
\log(2-2\theta)-\frac12>0.
\]
Next choose $\kappa>0$ sufficiently small in terms of $\theta$, and finally decrease $\eta$ whenever required below.

Suppose first that $\omega\geq\kappa$. Equations \eqref{eq:component-mass-lower} and \eqref{eq:s-positive} give
$\mu-C\geq s_0\kappa-\eta/2$.
We may assume that the right side is at least $\mu_0=s_0\kappa/2$. Since $d_{p,J}(i)\leq C$ and $\mu\leq C^2\leq4$, we have
\[
d_{p,J}(i)\leq \mu-\mu_0\leq\left(1-\frac{\mu_0}{4}\right)\mu.
\]
Moreover, Lemma \ref{lem:f-bound} and $\mu\geq C$ imply
\[
\Gamma(p,c;J)\geq\log C-\frac{C-1}{2}
\geq \min_{1+s_0\leq x\leq2}\left(\log x-\frac{x-1}{2}\right)>0.
\]
Thus $w=p[J]$ has both required margins in this case.

We may now suppose that $\omega<\kappa$. Equation \eqref{eq:component-mass-lower} gives
\begin{equation}\label{eq:mass-near-two}
\mu\geq2-\kappa^2-\frac\eta2.
\end{equation}
Let
\[
U=\{i\in J:d_{p,J}(i)>(1-\theta)\mu\}.
\]
If $U$ is empty, take $w=p[J]$. Equation \eqref{eq:capacity-margin} holds with $\sigma=\theta$, and Lemma \ref{lem:f-bound} gives
\[
\Gamma(p,c;J)\geq\log\left(2-\kappa^2-\frac\eta2\right)-\frac12>0
\]
by the choices of the parameters.

Suppose next that $U$ is nonempty and $c(U)<1/4$. Let $\alpha=\theta/4$ and multiply every $p$-weight having at least one end in $U$ by $1-\alpha$. Denote the resulting matrix on $J$ by $w$. Let $M_U$ be the ordered $p$-mass of pairs having at least one end in $U$. Symmetry gives
\begin{equation}\label{eq:incident-mass}
M_U\leq2\sum_{i\in U}c_i d_{p,J}(i)\leq2c(U)C<\frac C2.
\end{equation}
Thus
\begin{equation}\label{eq:scaled-mass}
\mu_w(J)=\mu-\alpha M_U\geq\mu-\frac{\alpha C}{2}.
\end{equation}
If $i\in U$, then
$d_{w,J}(i)=(1-\alpha)d_{p,J}(i)$, and
$d_{p,J}(i)\le C$ together with \eqref{eq:incident-mass} yields
\begin{align*}
\mu_w(J)-d_{w,J}(i)
&=\mu-d_{p,J}(i)-\alpha\bigl(M_U-d_{p,J}(i)\bigr)\\
&\geq\mu-C+\frac{\alpha C}{2}
\geq-\frac\eta2+\frac\alpha2.
\end{align*}
If $i\notin U$, then
\[
\mu_w(J)-d_{w,J}(i)
\geq\theta\mu-\frac{\alpha C}{2}.
\]
After decreasing $\eta$ and using \eqref{eq:mass-near-two}, both differences are at least $\beta_0=\theta/16$. Since $\mu_w(J)\leq4$, condition \eqref{eq:capacity-margin} follows with $\sigma=\beta_0/4$. Equations \eqref{eq:scaled-mass} and \eqref{eq:gamma-lower} give
\[
\Gamma(w,c;J)\geq
\log\left(2-\kappa^2-\frac\eta2-\alpha\right)-\frac12>0.
\]
Every positive entry of $w$ is at least $(1-\alpha)$ times the corresponding entry of $p$.

It remains to consider $c(U)\geq1/4$. We use the other color. For $i\in U$, the fact that $J$ is a component of $G_p$ and $p_{ij}+q_{ij}\leq1$ gives
\[
\sum_{j\in I}c_jq_{ij}\leq2-d_{p,J}(i)
<2\theta+\kappa^2+\frac\eta2=:\xi.
\]
As $U$ is nonempty, $\mu<C/(1-\theta)$. The $p$-mass outside $J$ is at most $\omega^2$. Since $0\leq\omega<\kappa<1$,
\[
R_p<\frac{2-\omega}{1-\theta}+\omega^2
\leq\frac{2}{1-\theta}\leq2+5\theta,
\]
where the last inequality follows from the choice of $\theta$. Consequently,
\[
R_q=4-\Delta(p,q)-R_p\geq2-5\theta-\eta.
\]
The ordered $q$-mass of pairs with an end in $U$ is at most $2\sum_{i\in U}c_i\sum_jc_jq_{ij}\leq4\xi$. If $W=I\setminus U$ and $\mu_q(W)=\sum_{i,j\in W}c_ic_jq_{ij}$, then
\begin{equation}\label{eq:qW-mass}
\mu_q(W)>2-13\theta-4\kappa^2-3\eta>\frac{15}{8}.
\end{equation}
Also $c(W)\leq7/4$. Every bipartite component of $G_q[W]$, and every non-bipartite component of capacity at most one, has mass at most its capacity. Hence \eqref{eq:qW-mass} forces a non-bipartite component $L$ with
\begin{equation}\label{eq:q-component-surplus}
c(L)>1
\quad \text{and} \quad
\mu_q(L)\geq c(L)+\frac18.
\end{equation}
Indeed, there is at most one component of capacity greater than one, and the other components together have mass at most their total capacity. Since $c(L)\leq7/4$, we have
\[
d_{q,L}(i)\leq c(L)\leq\mu_q(L)-\frac18
\leq\left(1-\frac1{32}\right)\mu_q(L).
\]
Lemma \ref{lem:f-bound} and \eqref{eq:q-component-surplus} imply
\[
\Gamma(q,c;L)\geq
\log\left(c(L)+\frac18\right)-\frac{c(L)-1}{2}
\geq\log\frac98>0.
\]
The second inequality follows since $1<c(L)\leq7/4$ and the function $h(x)=\log(x+1/8)-(x-1)/2$ is increasing on $[1,7/4]$. In this case, take $J=L$ and let $w$ be the restriction of $q$ to $L\times L$, extended by zero outside $L\times L$. Under the final hypothesis of the lemma, every positive entry of $w$ is at least $\tau$.

Taking the minimum of the positive margins obtained in the four cases defines $\sigma$ and $\gamma$. We may take $a_0=1-\theta/4$. The dependence of all choices is only on $\delta$.
\end{proof}

\section{Colorings far from the split construction}
\subsection{Counting index words}
Fix $J$ and $w$ as in Lemma \ref{lem:robust-core}. In this subsection, write
\[
d_i=d_{w,J}(i),
\quad
\mu=\mu_w(J),
\quad
y_{ij}=\frac{c_ic_jw_{ij}}{\mu},
\quad \text{and} \quad
x_i=\frac{c_id_i}{\mu}.
\]
Here $y_{ij}$ is the desired proportion of steps from cluster $i$ to
cluster $j$, while $x_i$ is the desired proportion of visits to cluster
$i$.
Consequently,
\begin{equation}\label{eq:marginals}
\sum_{i,j\in J}y_{ij}=1,
\quad
\sum_{j\in J}y_{ij}=x_i,
\quad \text{and} \quad
\sum_{i\in J}x_i=1.
\end{equation}
Thus $y=(y_{ij})$ describes a balanced flow whose vertex marginals are
given by $x=(x_i)$. All sums and products involving $\log w_{ij}$ are restricted to pairs for which $w_{ij}>0$.

The following formula for Euler tours is commonly known as the BEST theorem \cite{AardenneBruijn,SmithTutte}. Its general directed form was proved by van Aardenne-Ehrenfest and de Bruijn, while Smith and Tutte had earlier obtained a special case. We use the version in which the first arc is prescribed. For a directed multigraph $\mathcal D$ and a vertex $v$, let $d_{\mathcal D}^{+}(v)$ denote the out-degree of $v$ in $\mathcal D$.

\begin{theorem}[van Aardenne-Ehrenfest and de Bruijn \cite{AardenneBruijn}]\label{thm:best}
Let $\mathcal D$ be a finite strongly connected Eulerian directed multigraph whose parallel arcs are distinguished. Fix a vertex $r$ and an arc $e_0$ leaving $r$. Let $t_r(\mathcal D)$ denote the number of spanning in-arborescences rooted at $r$, where an in-arborescence rooted at $r$ is a spanning directed tree in which every vertex has a directed path to $r$. Euler tours are counted as linear sequences of distinguished arcs. Then the number of Euler tours of $\mathcal D$ beginning with $e_0$ is
\[
t_r(\mathcal D)\prod_{v\in V(\mathcal D)}
\bigl(d_{\mathcal D}^+(v)-1\bigr)!.
\]
\end{theorem}

\begin{lemma}\label{lem:euler-words}
For every sufficiently large odd integer $k$, there is a family $\mathcal I_k$ of closed index words $\boldsymbol{i}=(i_1,\ldots,i_k)$ in $J$ with the following properties. Consecutive indices, including $i_k,i_1$, form edges of $G_w$. There are integers $n_{ij}$ and $b_i$ such that every word has exactly $n_{ij}$ transitions from $i$ to $j$ and contains $i$ exactly $b_i$ times, where
\begin{equation}\label{eq:transition-asymptotics}
n_{ij}=ky_{ij}+O(1),
\quad
b_i=kx_i+O(1).
\end{equation}
Moreover,
\begin{equation}\label{eq:word-count}
|\mathcal I_k|\geq
\frac{\prod_{i\in J}(b_i-1)!}{\prod_{i,j\in J}n_{ij}!}
\geq \exp\bigl(k(H(y)-H(x))-o(k)\bigr),
\end{equation}
where $H(x)=-\sum_i x_i\log x_i$ and $H(y)=-\sum_{i,j}y_{ij}\log y_{ij}$.

If the positive integers $M_i$ satisfy $M_i=c_i k+O(1)$ and $x_i\leq(1-\sigma)c_i$ for every $i\in J$, define
\begin{equation}\label{eq:ideal-word-weight}
\mathcal W(\boldsymbol{i})=
\prod_{t=1}^k w_{i_ti_{t+1}}
\prod_{i\in J}(M_i)_{b_i},
\quad i_{k+1}=i_1.
\end{equation}
Then
\begin{equation}\label{eq:ideal-total}
\sum_{\boldsymbol{i}\in\mathcal I_k}\mathcal W(\boldsymbol{i})
\geq k!\exp\bigl(k\Gamma(w,c;J)-o(k)\bigr).
\end{equation}
\end{lemma}

\begin{proof}
Choose an odd cycle $F$ in the connected non-bipartite graph
$G_w[J]$, write $|F|$ for its length, and let $h$ be the number
of edges of $G_w[J]$. For an undirected edge $e=ij$, write
$y_e=y_{ij}$. Symmetry and \eqref{eq:marginals} give
$\sum_e y_e=1/2$. Define
\[
T_e=ky_e-\frac{|F|}{2h}.
\]
Since $k$ and $|F|$ are odd and $\sum_e y_e=1/2$, we have
$\sum_e T_e=(k-|F|)/2\in\mathbb Z$. Consequently,
\[
\frac{k-|F|}{2}-\sum_e\lfloor T_e\rfloor
=
\sum_e\bigl(T_e-\lfloor T_e\rfloor\bigr)
\]
is an integer in $\{0,\ldots,h-1\}$. Starting with the integers
$\lfloor T_e\rfloor$, increase by one exactly this many terms with
the largest fractional parts, and denote the resulting integers by
$a_e$. Since the support graph is finite and $y_e>0$ for every
edge $e$, we have $T_e\ge 1$ for every $e$ when $k$ is sufficiently
large. Hence every $a_e$ is a positive integer. By construction and
the definition of $T_e$, we have
\begin{equation}\label{eq:ae}
\sum_e a_e=\frac{k-|F|}{2}
\quad\text{and}\quad
a_e=ky_e+O(1).
\end{equation}

Construct a directed multigraph $\mathcal D_k$ on $J$. 
For every edge $e=ij$, include $a_e$ arcs in each direction between $i$ and $j$. 
Direct the fixed odd cycle $F$ cyclically and add one further arc along each of its edges. 
The number of arcs is $2\sum_ea_e+|F|=k$. Every vertex has equal in-degree and out-degree. The symmetric arcs make $\mathcal D_k$ strongly connected, so it is Eulerian.

Let $n_{ij}$ be the number of arcs from $i$ to $j$ and let $b_i$ be the out-degree of $i$. Equation \eqref{eq:ae} gives \eqref{eq:transition-asymptotics}. Label all parallel arcs. Fix a vertex $r_*$ and one labeled arc $e_*$ from $r_*$ to $s_*$. By Theorem \ref{thm:best}, the number of labeled Euler tours beginning with $e_*$ is
\[
t_{r_*}(\mathcal D_k)\prod_{i\in J}(b_i-1)!,
\]
where $t_{r_*}(\mathcal D_k)\geq1$ is the number of in-arborescences rooted at $r_*$. After the labels on parallel arcs are forgotten, each resulting closed index word corresponds to exactly
\[
(n_{r_*s_*}-1)!\prod_{(i,j)\neq(r_*,s_*)}n_{ij}!
=\frac{\prod_{i,j}n_{ij}!}{n_{r_*s_*}}
\]
labeled tours. Hence the first inequality in \eqref{eq:word-count} follows. Stirling's formula and \eqref{eq:transition-asymptotics} yield
\begin{align*}
\log|\mathcal I_k|
\geq \sum_i\log(b_i-1)!-\sum_{i,j}\log n_{ij}!
\geq k\sum_i x_i\log x_i-k\sum_{i,j}y_{ij}\log y_{ij}-o(k),
\end{align*}
which is the second inequality in \eqref{eq:word-count}.

All words in $\mathcal I_k$ have the same transition and occurrence counts. The positive weights are fixed, so the $O(1)$ transition-rounding errors contribute only $O(1)$ after logarithms are taken. Using \eqref{eq:ideal-word-weight} and \eqref{eq:word-count}, we obtain
\begin{align}\label{eq:ideal-log-start}
\log\sum_{\boldsymbol{i}\in\mathcal I_k}\mathcal W(\boldsymbol{i})
\geq{}
k(H(y)-H(x))+k\sum_{i,j}y_{ij}\log w_{ij} +\sum_i\log(M_i)_{b_i}-o(k).
\end{align}
The capacity margin implies $M_i-b_i$ is linear in $k$. Stirling's formula gives
\begin{equation}\label{eq:falling-entropy}
\log(M_i)_{b_i}
=kx_i\log k+k\bigl(c_i\log c_i-(c_i-x_i)\log(c_i-x_i)-x_i\bigr)+o(k).
\end{equation}
Also,
\begin{equation}\label{eq:entropy-identity}
H(y)+\sum_{i,j}y_{ij}\log w_{ij}
=-2\sum_i x_i\log c_i+\log\mu.
\end{equation}
Substitute \eqref{eq:falling-entropy} and \eqref{eq:entropy-identity} into \eqref{eq:ideal-log-start}. Writing $z_i=d_i/\mu$, so that $x_i=c_i z_i$, cancels all terms containing $\log c_i$. The coefficient of $k$ after $k\log k$ is
\[
-1+\log\mu+\sum_i c_i\bigl(z_i\log z_i-(1-z_i)\log(1-z_i)\bigr)
=-1+\Gamma(w,c;J).
\]
Since $\log k!=k\log k-k+o(k)$, inequality \eqref{eq:ideal-total} follows.
\end{proof}

\subsection{Embedding cycles through regular pairs}
To obtain the regular partition used in the far-from-split argument and to lift the closed index words constructed above to simple cycles, we recall the standard notion of an $\varepsilon$-regular pair and the equitable form of Szemerédi's regularity lemma \cite{Szemeredi,KomlosSimonovits}. Let $G$ be a graph. For disjoint nonempty sets $X,Y\subseteq V(G)$, let $e_G(X,Y)$ denote the number of edges joining $X$ and $Y$, and write $d_G(X,Y)=e_G(X,Y)/(|X||Y|)$. The pair $(X,Y)$ is $\varepsilon$-regular in $G$ if $|d_G(X',Y')-d_G(X,Y)|\leq\varepsilon$ whenever $X'\subseteq X$ and $Y'\subseteq Y$ satisfy $|X'|\geq\varepsilon|X|$ and $|Y'|\geq\varepsilon|Y|$. For $v\in V(G)$, let $N_G(v)$ denote the neighborhood of $v$ in $G$. When the underlying graph is clear, we omit the subscript $G$. A partition is equitable if the sizes of any two parts differ by at most one.

\begin{lemma}[Szemerédi \cite{Szemeredi}]\label{lem:regularity}
For every $\varepsilon>0$ and positive integer $L$, there exist positive integers $M_0=M_0(\varepsilon,L)$ and $n_0=n_0(\varepsilon,L)$ such that every graph $G$ with $|V(G)|\geq n_0$ has an equitable partition $V(G)=V_1\cup\cdots\cup V_M$, where $L\leq M\leq M_0$, for which all but at most $\varepsilon M^2$ unordered pairs $(V_i,V_j)$ with $1\leq i<j\leq M$ are $\varepsilon$-regular.
\end{lemma}

In a red/blue edge-coloring, a pair is $\varepsilon$-regular in red if and only if it is $\varepsilon$-regular in blue.

\begin{lemma}[see {\cite[Fact 1.3]{KomlosSimonovits}}]\label{lem:typical-vertex}
Let $(X,Y)$ be an $\varepsilon$-regular pair of density $P$, and let $Y'\subseteq Y$ satisfy $|Y'|\geq\varepsilon|Y|$. Apart from at most $\varepsilon|X|$ vertices $x\in X$, every vertex satisfies
\[
|N(x)\cap Y'|\geq(P-\varepsilon)|Y'|.
\]
A vertex satisfying this inequality is called typical with respect to $Y'$.
\end{lemma}

\begin{lemma}\label{lem:lifting}
Fix $\sigma,\tau,\nu>0$. There are $\varepsilon_0>0$ and positive
integers $k_0$ and $n_0$ such that the following holds whenever
$0<\varepsilon\le\varepsilon_0$, $k\ge k_0$, and $|V_i|\ge n_0$
for every $i\in[M]$. Let $G$ be a graph, and let $V_1,\ldots,V_M$ be pairwise disjoint subsets of $V(G)$. For every edge $ij$ of a fixed support graph on $[M]$, suppose that $(V_i,V_j)$ is an $\varepsilon$-regular pair in $G$, and write $P_{ij}=d_G(V_i,V_j)$. Assume that
$P_{ij}\geq w_{ij}\geq\tau$.
Let $\boldsymbol{i}=(i_1,\ldots,i_k)$ be a closed index word in the support graph, and let $b_i$ be the number of occurrences of $i$. If
\begin{equation}\label{eq:lifting-capacity}
b_i\leq(1-\sigma)|V_i|
\quad\text{for every }i,
\end{equation}
then the number of injective sequences $v_t\in V_{i_t}$ satisfying $v_tv_{t+1}\in E(G)$ for $1\leq t\leq k$, with $v_{k+1}=v_1$, is at least
\begin{equation}\label{eq:lifting-bound}
e^{-\nu k}
\prod_{t=1}^k w_{i_ti_{t+1}}
\prod_i(|V_i|)_{b_i}.
\end{equation}
\end{lemma}

\begin{proof}
Replacing $\sigma$ by $\min\{\sigma,1\}$ does not weaken the assertion. Choose $\beta>0$ so small that $\beta<\min\{\sigma/4,\tau/4\}$ and $\beta/(\sigma\tau)<\min\{1/8,\nu/8\}$. Next choose $\varepsilon_0>0$ so small that every $0<\varepsilon\leq\varepsilon_0$ satisfies $\varepsilon<\beta/2$ and
\[
\rho:=\frac{\beta+2\varepsilon}{\sigma\tau}<\min\left\{\frac14,\frac\nu4\right\}.
\]
These choices are uniform over all words and clusters.

Select $v_1\in V_{i_1}$ that is typical with respect to both $V_{i_2}$ and $V_{i_k}$. There are at least $(1-2\varepsilon)|V_{i_1}|$ choices. For every such $v_1$, typicality with respect to $V_{i_k}$ and the inequalities $P_{i_ki_1}\geq w_{i_ki_1}\geq\tau$ imply
\[
|N_G(v_1)\cap V_{i_k}|
\geq(\tau-\varepsilon)|V_{i_k}|.
\]
Since $\varepsilon<\beta/2$ and $\beta<\tau/4$, this quantity is greater than $\beta|V_{i_k}|$. Thus one may choose
$B_{\mathrm{buf}}\subseteq N(v_1)\cap V_{i_k}$ with
$|B_{\mathrm{buf}}|=\lfloor\beta|V_{i_k}|\rfloor$.
Lemma \ref{lem:typical-vertex}, applied to the pair $(V_{i_{k-1}},V_{i_k})$ and the set $B_{\mathrm{buf}}$, allows us to choose
$A_{\mathrm{buf}}\subseteq V_{i_{k-1}}$ with 
$|A_{\mathrm{buf}}|=\lfloor\beta|V_{i_{k-1}}|\rfloor$,
so that every vertex of $A_{\mathrm{buf}}$ has at least $(P_{i_{k-1}i_k}-\varepsilon)|B_{\mathrm{buf}}|$ neighbors in $B_{\mathrm{buf}}$. If $i_{k-1}=i_1$, choose $A_{\mathrm{buf}}$ to avoid $v_1$. None of the positions from $2$ through $k-2$ will use either buffer.

Before position $t$, let
$R_t^*=V_{i_t}\setminus\{v_s:1\leq s<t,\ i_s=i_t\}$. Condition \eqref{eq:lifting-capacity} gives $|R_t^*|\geq\sigma|V_{i_t}|$. For $2\leq t\leq k-2$, let $R_t$ be the unused vertices of $V_{i_t}$ after the relevant buffer has also been excluded. Then
\begin{equation}\label{eq:available-set}
|R_t|\geq |R_t^*|-\beta|V_{i_t}|
\geq\left(1-\frac\beta\sigma\right)|R_t^*|
\geq\frac\sigma2|V_{i_t}|.
\end{equation}

We choose the vertices in order. At position $2$, the typicality of $v_1$ with respect to $V_{i_2}$ leaves at least $(P_{i_1i_2}-\varepsilon)|V_{i_2}|-\beta|V_{i_2}|$ neighbors after the buffer is removed. Discard at most $\varepsilon|V_{i_2}|$ vertices that are not typical with respect to the next available set. Since $R_2^*=V_{i_2}$, the number of choices for $v_2$ is at least $(1-\rho)w_{i_1i_2}|R_2^*|$.
The chosen vertex has at least $(P_{i_2i_3}-\varepsilon)|R_3|$ neighbors in $R_3$.

For $3\leq t\leq k-2$, assume that $v_{t-1}$ has at least
$(P_{i_{t-1}i_t}-\varepsilon)|R_t|$ neighbors in $R_t$.
We choose $v_t$ from these neighbors, retaining only those that are
typical with respect to the available set for position $t+1$.
This set is $R_{t+1}$ when $t<k-2$ and $A_{\mathrm{buf}}$ when
$t=k-2$.
Because $i_t\ne i_{t+1}$, the next available set is determined
before $v_t$ is chosen.
In either case, its size is at least $\varepsilon|V_{i_{t+1}}|$,
so by Lemma \ref{lem:typical-vertex}, at most
$\varepsilon|V_{i_t}|$ candidates are excluded. Equation \eqref{eq:available-set} gives at least
\begin{align*}
(P_{i_{t-1}i_t}-\varepsilon)|R_t|-\varepsilon|V_{i_t}|
\geq w_{i_{t-1}i_t}|R_t^*|-(\beta+2\varepsilon)|V_{i_t}|
\geq(1-\rho)w_{i_{t-1}i_t}|R_t^*|
\end{align*}
choices. The selected vertex is typical with respect to the next set, so the induction continues.

At position $k-1$, choose a neighbor of $v_{k-2}$ in $A_{\mathrm{buf}}$. For sufficiently large clusters and $\varepsilon\leq\tau/2$, there are at least
\[
\frac\beta4 w_{i_{k-2}i_{k-1}}|R_{k-1}^*|
\]
choices. Every chosen vertex has at least $(P_{i_{k-1}i_k}-\varepsilon)|B_{\mathrm{buf}}|$ neighbors in $B_{\mathrm{buf}}$, so there are at least
\[
\frac\beta4 w_{i_{k-1}i_k}|R_k^*|
\]
choices for $v_k$. The edge $v_kv_1$ is present because $B_{\mathrm{buf}}\subseteq N(v_1)$. Since $w_{i_ki_1}\leq1$, including this final weight only decreases the desired lower bound.

Finally,
\[
|V_{i_1}|\prod_{t=2}^k|R_t^*|=\prod_i(|V_i|)_{b_i}.
\]
Multiplying all estimates gives at least
\[
\frac12\left(\frac\beta4\right)^2(1-\rho)^{k-3}
\prod_{t=1}^k w_{i_ti_{t+1}}
\prod_i(|V_i|)_{b_i}
\]
sequences. Since $\rho<\min\{1/4,\nu/4\}$, the inequality $\log(1-\rho)\geq-2\rho$ gives $(1-\rho)^k\geq e^{-\nu k/2}$. The fixed prefactor is at least $e^{-\nu k/2}$ for sufficiently large $k$. This proves \eqref{eq:lifting-bound}.
\end{proof}

\subsection{The far-from-split bound}
We now combine the preceding results to establish the following lower bound for colorings that are not $\lambda$-near-split.
\begin{prop}\label{prop:far-split}
For every $\lambda>0$ and every fixed nonnegative integer $\ell$, there are constants $g=g(\lambda,\ell)>0$ and $k_1=k_1(\lambda,\ell)$ 
such that every odd $k\geq k_1$ and every red/blue edge-coloring $\chi$ of $K_{2k-1+\ell}$ that is not $\lambda$-near-split satisfies
\begin{equation}\label{eq:far-split-gain}
\Nmono(C_k,\chi)\geq e^{gk}H_m>H_m.
\end{equation}
\end{prop}

\begin{proof}
Fix $\lambda>0$ and $\ell$. Choose $0<\delta<1/2$ such that
$\delta/(1-\delta)^2<\lambda$. Apply Lemma \ref{lem:robust-core} with this $\delta$ and obtain $\eta,\gamma,\sigma,a_0$. Choose $0<\tau<1/2$ so that $4\tau<\eta/4$, and let $\nu=\gamma/4$. Apply Lemma \ref{lem:lifting} with parameters $\sigma/2$, $a_0\tau$, and $\nu$. Choose $\varepsilon>0$ below the resulting threshold and then choose $L$ sufficiently large that
\begin{equation}\label{eq:regularity-parameters}
20\varepsilon+4\tau+\frac5L<\eta.
\end{equation}
These choices depend only on $\lambda$ and precede the choice of $k$.

Apply Lemma \ref{lem:regularity} to the red graph and write the resulting equitable partition as $V_1\cup\cdots\cup V_M$. Let
\[
c_i=\frac{2|V_i|}{N},
\]
so that $\sum_i c_i=2$. For a regular pair $(V_i,V_j)$, let $p_{ij}$ equal its red density when that density is at least $\tau$, and let $p_{ij}=0$ otherwise. Define $q_{ij}$ in the same way from the blue density. Assign zero to both entries on irregular pairs and on the diagonal. Every positive entry is at least $\tau$.

We estimate $\Delta(p,q)$. For sufficiently large $N$,
equitability gives $c_i=2/M+O(1/N)$ uniformly in $i$. The irregular pairs contribute less than $20\varepsilon$, and the regular off-diagonal pairs contribute at most $4\tau$. The diagonal pairs contribute $\sum_i c_i^2=4/M+O(1/N)<5/M\le 5/L$. Therefore,
\begin{equation}\label{eq:template-small-defect}
\Delta(p,q)<20\varepsilon+4\tau+\frac{5}{L}<\eta.
\end{equation}
Suppose that $D_{p,q}(X)<\delta$, and let $A=\bigcup_{i\in X}V_i$ and $B=\bigcup_{i\notin X}V_i$. Since $|c(X)-1|<\delta$, both parts have order at least $(1-\delta)N/2$. Since $p_{ii}=0$ and $p_{ij}$ does not exceed the red density between $V_i$ and $V_j$ for $i\ne j$, the number of missing ordered red edges in $A$ is at most
\[\frac{N^2}{4}\sum_{i,j\in X}c_ic_j(1-p_{ij})
<\frac{\delta N^2}{4}.
\]
This estimate also accounts for the diagonal pairs, so
$|A|^2-2e_R(A)<\delta N^2/4$.
Thus $d_R(A)\geq 2e_R(A)/|A|^2>1-\delta/(1-\delta)^2$,
and the same lower bound holds for $d_R(B)$.
Similarly, since the cross term in $D_{p,q}(X)$ includes both
ordered directions,
\[
e_R(A,B)
\leq \frac{N^2}{4}
\sum_{\substack{i\in X\\j\notin X}}c_ic_j(1-q_{ij})
<\frac{\delta N^2}{8}.
\]
Consequently, $d_B(A,B)>1-\delta/[2(1-\delta)^2]$.
By the choice of $\delta$, the coloring is $\lambda$-near-split. The same conclusion, after exchanging the colors, follows from $D_{q,p}(X)<\delta$. Since the coloring in the proposition is not $\lambda$-near-split, we have
\begin{equation}\label{eq:template-far}
\spl(p,q)\geq\delta.
\end{equation}

Lemma \ref{lem:robust-core}, together with
\eqref{eq:template-small-defect} and \eqref{eq:template-far},
provides one color, weights $w$, and a connected non-bipartite
component $J$ of $G_w$ satisfying
\[
d_{w,J}(i)\leq(1-\sigma)\mu_w(J),
\quad
\Gamma(w,c;J)\geq\gamma,
\quad \text{and} \quad
w_{ij}>0\Longrightarrow w_{ij}\geq a_0\tau.
\]
The cluster sizes satisfy
\[
|V_i|=\frac{c_iN}{2}=c_i\left(k+\frac{\ell-1}{2}\right)=c_i k+O_{\ell}(1).
\]
Since $M\leq M_0$, equitability bounds every $c_i$ away from zero, and every positive $w_{ij}$ is at least $a_0\tau$. The error terms and the threshold for $k$ in Lemma \ref{lem:euler-words} are therefore uniform over all templates arising here. For the occurrence counts constructed in that lemma,
\[
b_i=k\frac{c_i d_{w,J}(i)}{\mu_w(J)}+O(1)
\leq(1-\sigma)c_i k+O(1)
\leq\left(1-\frac\sigma2\right)|V_i|
\]
for sufficiently large $k$ in terms of $\lambda$ and $\ell$. 
The hypothesis of Lemma~\ref{lem:euler-words} also holds with $M_i=|V_i|$, since $x_i=c_i d_{w,J}(i)/\mu_w(J)\le(1-\sigma)c_i$. 
It gives an ideal total weight at least 
\[
k!\exp\bigl((\gamma-o(1))k\bigr).
\]
Every positive support pair is regular in the chosen color and has actual density at least $w_{ij}\geq a_0\tau$. Applying Lemma \ref{lem:lifting} to every word, with loss $e^{-\gamma k/4}$, yields more than
\[
k!\exp\bigl((3\gamma/4-o(1))k\bigr)
\]
rooted directed simple monochromatic cycles. Distinct index words produce disjoint sets of vertex sequences because the regularity partition determines every index. Each unrooted unoriented $C_k$ has exactly $2k$ rooted directed representations. Hence
\[
\Nmono(C_k,\chi)
\geq H_m\exp\bigl((3\gamma/4-o(1))k\bigr).
\]
Taking $g=\gamma/2$ and increasing $k_1$ proves \eqref{eq:far-split-gain}.
\end{proof}

\section{Proof of Theorem \ref{thm:main}}
We now combine the preceding results to determine the exact value
of $M(C_k,2k-1+\ell)$ and characterize all colorings attaining
equality.

\begin{proof}[\bfseries{Proof of Theorem~\ref{thm:main}}]
Fix $\ell$, and let $\lambda_*$ be the constant in Proposition~\ref{prop:near-split}. 
Consider a red/blue edge-coloring $\chi$ of $K_N$. If it is not $\lambda_*$-near-split, then Proposition~\ref{prop:far-split} yields at least $e^{gk}H_m$ monochromatic copies of $C_k$, which is more than $F_\ell(k)$ by Lemma~\ref{lem:estimate-Fl(k)}.
If it is $\lambda_*$-near-split and already has more than $F_\ell(k)$ copies, there is nothing to prove. 
Otherwise, Proposition~\ref{prop:near-split} gives a partition $V(K_N)=P\cup Q$ into two cliques of one color, say red. Write $p=|P|$ and $q=|Q|$.

Every $k$-vertex subset of a red clique spans exactly $H_m$ Hamilton cycles. 
Thus the cycles wholly inside $P$ or $Q$ already number $(\binom pk+\binom qk)H_m$, which is at least $F_\ell(k)$ by Lemma~\ref{lem:balance}. 
This proves the lower bound. Both $\chi_0(\alpha,\beta)$ and $\chi_1(\alpha,\beta)$ attain this bound by the construction in the introduction.

Now suppose that equality holds, that is $\Nmono(C_k,\chi)=F_\ell(k)$. 
The case in which it is not $\lambda_*$-near-split is impossible,
and Proposition~\ref{prop:near-split} again gives a partition into red cliques $P,Q$. 
Lemma~\ref{lem:balance} forces $\{p,q\}=\{\alpha,\beta\}$, and the cycles inside the two cliques already account for all $F_\ell(k)$ copies. In particular, $p,q\ge k-1$.

There is at most one red edge between $P$ and $Q$. Suppose first that $p_1q_1$ and $p_2q_2$ are red edges with four distinct ends, where $p_i\in P$ and $q_i\in Q$. Choose a red path of length $k-3$ from $q_1$ to $q_2$ in $Q$. Such a path uses $k-2$ vertices and exists because $|Q|\ge k-1$. Together with $p_1q_1$, $q_2p_2$, and $p_2p_1$, it forms a red $C_k$ meeting both parts. If two red cross-edges instead have a common end in $P$, join their ends in $Q$ by a red path of length $k-2$. Adding the common end again forms a red $C_k$ meeting both parts. The case of a common end in $Q$ is symmetric. Each case produces an additional cycle, contrary to equality.

All cross-edges are consequently blue, with at most one exception. 
With no exception, $\chi$ is equivalent to $\chi_0(\alpha,\beta)$; with one exception, it is equivalent to $\chi_1(\alpha,\beta)$. 
This completes the proof.
\end{proof}

\section*{Declarations}
The authors declare that they have no known competing financial interests or personal relationships that could have appeared to influence the work reported in this paper.

\section*{Availability of Data and Materials}
Not applicable.

\section*{Acknowledgments}
This research was supported by the National Key R\&D Program of China under grant number 2024YFA1013900, the National Natural Science Foundation of China under grant number 12471327, and the China Postdoctoral Science Foundation under grant number 2026M793375.

\end{document}